\documentclass[a4paper,reqno]{amsart}

\usepackage[english]{babel}
\usepackage{amssymb,amsmath,amsthm,bbm}

\makeatletter 
\def\th@plain{%
\thm@notefont{\bfseries}
   \itshape 
} 
\makeatother

\theoremstyle{plain}
\newtheorem{theorem}{Theorem}[section]
\newtheorem{lemma}[theorem]{Lemma}
\newtheorem{corollary}[theorem]{Corollary}
\newtheorem{proposition}[theorem]{Proposition}

\theoremstyle{definition}
\newtheorem{definition}[theorem]{Definition}
      
\theoremstyle{remark}
\newtheorem{remark}[theorem]{Remark}
\newtheorem{example}[theorem]{Example}

\numberwithin{equation}{section}

\newcommand{\res}{\mathop{\hbox{\vrule height 7pt width .5pt depth 0pt
\vrule height .5pt width 6pt depth 0pt}}\nolimits}
\newcommand{\per}[2]{|\partial {#1}|_{\mathbb G}({#2})}
\newcommand{\Per}[1]{|\partial {#1}|_{\mathbb G}}
\newcommand{\scal}[3]{\langle
{#1} , {#2}\rangle_{#3}}
\newcommand{\Scal}[2]{\langle {#1} , {#2}\rangle}
\newcommand{\average}{{\mathchoice {\kern1ex\vcenter{\hrule height.4pt
width 6pt depth0pt} \kern-9.7pt} {\kern1ex\vcenter{\hrule height.4pt width
4.3pt depth0pt}
\kern-7pt} {} {} }}
\newcommand{\ave}{\average\int}

\begin{document}

\title[Constant intrinsic normal sets in a class of Carnot groups]{Regularity of sets with constant intrinsic normal in a class of Carnot groups}
\author{Marco Marchi}
\address{Marco Marchi: Dipartimento di Matematica (Universit\`a degli Studi di Milano) \newline \indent
via Cesare Saldini 50\\ 20133 Milano MI\\ Italy} \email{marco.marchi@unimi.it}
\thanks{It is a pleasure to thank Bruno Franchi for his invaluable constant support and for the many hours he spent talking with the author, and Alessandro Ottazzi for several helpful discussions on Carnot groups of type $\star$. The author also thanks the referee for his useful comments which helped improving the presentation of the paper.}
\date{\today}

\begin{abstract}
In this Note, we define a class of stratified Lie groups of arbitrary step (that are called ``groups of type $\star$''
throughout the paper), and we prove that, in these groups, sets with constant intrinsic normal are vertical halfspaces. As a consequence, the reduced boundary of a set of finite intrinsic perimeter
in a group of type $\star$ is rectifiable in the intrinsic sense (De Giorgi's rectifiability theorem). This result extends 
the previous one proved by Franchi, Serapioni \& Serra Cassano in step 2 groups.
\end{abstract}
\subjclass[2010]{28A75,  49Q15, 58C35}

\maketitle
\tableofcontents

\section{Introduction}
This article is intended to introduce a peculiar class of nilpotent Lie groups, whose Lie algebra enjoys a particular property giving several interesting algebraic and geometric consequences.
We call these groups \textquotedblleft of type $\star$\textquotedblright. As long as we know, these groups are considered for the first time in the present paper. Before entering into the details let us spend a few words to illustrate the motivation of our work. 

In fact our initial intent was to extend to a larger class of stratified nilpotent Lie groups (usually called Carnot groups nowadays) De Giorgi's Rectifiability Theorem following the original De Giorgi's approach (\cite{degiorgi}) that was partially extended to some Carnot groups by  \cite{K1} and \cite{newambro}. Given a set $E \subset \mathbb R^n$ of finite perimeter, De Giorgi's proof relies on the fact that the perimeter measure is concentrated in a subset of the topological boundary of $E$, the so-called reduced boundary $\partial^\ast E$, and that at any point of $\partial^\ast E$ the blowup of $E$ is a set with a constant normal and therefore, \textsl{in the Euclidean setting}, is a halfspace.   The notions of perimeter and of  reduced boundary, as well as that of the normal have a natural counterpart in Carnot groups. This notions depend on the stratification of the Lie algebra and will be presented in details in Section \ref{constantnormal}. From now on we refer to this notions as to intrinsic perimeter, intrinsic reduced boundary and intrinsic normal.

Again the blowup of a set $E$ at a point of its reduced boundary is a set with constant intrinsic normal. However sets with constant intrinsic normal in general Carnot groups may fail to be halfspaces. In fact this result still holds in step 2 Carnot groups as it is proved in \cite{K1}, but there are counterexamples in groups of step 3 as shown in Example 3.2 of \cite{K1}. On the other hand,
in Carnot groups of arbitrary step,
 only a partial result is known, that has been proved in \cite{newambro} by Ambrosio,
 Kleiner \& Le Donne. 
In fact, the authors show that, given a set $E$ of locally finite perimeter in a Carnot group $\mathbb G$, 
then for almost every $x \in \mathbb G$ (with respect to the perimeter measure of $E$), some 
blow-up of $E$ at $x$ is a vertical halfspace.
The main result of the present paper holds precisely that constant intrinsic normal sets in Carnot groups of type $\star$ are halfspaces.
Thus in particular De Giorgi's Rectifiability Theorem holds in groups of type $\star$.

Let us sketch now the basic points of our paper.
Detailed definitions are given below: here we restrict ourself to
remind that a nilpotent, simply connected Lie group $\mathbb G$
is called a Carnot group if its Lie algebra ${\mathfrak{g}}$ admits the stratification
${\mathfrak{g}}=V_1\oplus \dots \oplus V_\kappa$, with
 $[V_1,V_i]=V_{i+1}$, where $[V_1,V_i]$ is the subspace of ${\mathfrak{g}}$ generated by 
 commutators $[X,Y]$ with $X\in V_1$ and $Y\in V_i$. The integer $\kappa$ is said the
 step of the group.

$\mathbb G$ is of type $\star$ if there exists a basis $(X_{1},\ldots,X_{m_{1}})$ of $V_{1}$
such that \begin{equation*}[X_{j},[X_{j},X_{i}]]=0 \, \mbox{ for } \, i,j=1,\ldots,m_{1}. \end{equation*}

Obviously groups of step 2 are of type $\star$. 
The most important examples of groups of type $\star$ are the nilpotent groups coming from the Iwasawa decomposition of $GL_{\kappa+1}(\mathbb R)$, i.e. the group of unit upper triangular matrices with 1's in the diagonal. In particular, there are groups of type $\star$ with arbitrarily large step. On the other hand there are Carnot groups that are not of type $\star$, i.e. free Carnot groups and filiform groups of step greater than 2.

In Section \ref{groupsstar} we define groups of type $\star$ and show some examples of them.
In Section \ref{constantnormal} we define the notions of intrinsic perimeter and intrinsic normal and we prove that, in such groups, sets with constant intrinsic normal are halfspaces (Theorem \ref{constant_halfspace}).
In Section \ref{geometricmotivation} we recall some results about rectifiability in Carnot groups from \cite{K1} and we enunciate the new results in groups of type $\star$ that follows straightforwardly from Theorem \ref{constant_halfspace}: Theorem \ref{blowup_new}, Theorem \ref{principale}, Theorem \ref{francesco3nuovo}, Corollary \ref{Balognuovo}.
In Section \ref{propriet} we provide some necessary conditions for a group to be of type $\star$ and some additional examples of groups in which the blowup of a set (at a point of its reduced boundary) is not always a halfspace.

\medskip

\section{Carnot groups of type $\star$} \label{groupsstar}
First we recall some preliminary definitions on Carnot groups.
For more details, one can see \cite{lanco_bonfi_ugu} and \cite{K1}.

\begin{definition}[Carnot group]
A Carnot group $\mathbb{G}$ of step $\kappa$ is a nilpotent connected and simply connected Lie group, whose Lie algebra
$\mathfrak{g}$ admit a step $\kappa$ stratification, i.e. there exist linear subspaces $V_1,\ldots,V_\kappa$ such that
\begin{equation}\label{stratificazione}
{\mathfrak{g}}=V_1\oplus \ldots \oplus V_\kappa,\quad [V_1,V_i]=V_{i+1},\quad
V_\kappa\neq\{0\},\quad V_i=\{0\}\;\textrm{ if } i>\kappa,
\end{equation}
where $[V_1,V_i]$ is the subspace of ${\mathfrak{g}}$ generated by commutators $[X,Y]$ with $X\in V_1$ and $Y\in V_i$. 
\end{definition}
We set
$m_i:=\dim(V_i)$ for $i=1,\dots,\kappa$. Obviously $\sum_{i=1}^\kappa m_i=n$, where $n$ is the dimension of $\mathfrak g$.

\bigskip

The left invariant vector fields that form a basis of $V_1$ are called generating vector fields of the group, since they generate $\mathfrak{g}$ 
for (\ref{stratificazione}). The choice of a basis $(X_1,\ldots,X_{m_1})$ for $V_1$ also fixes an inner product $\langle \cdot,\cdot\rangle$ on $V_1$ that makes $(\! X_1,\ldots,X_{m_1}\! )$ an orthonormal basis.
Precisely, if $Y=\sum_{j=1}^{m_1} y_j X_j$ and $Z=\sum_{j=1}^{m_1} z_j X_j$, then 
\[ \langle Y,Z \rangle :=\sum_{j=1}^{m_1} y_j z_j. \]
It is possible to chose a basis $(X_1,\ldots,X_n)$ for $\mathfrak g$ that is adapted to its stratification, that is
$ (X_{h_{j-1}+1},\ldots,X_{h_j})$  is a basis of  $V_j$  for  $1 \leq j \leq \kappa$, where   $h_j=\sum_{i=1}^j m_i$. 

Since Carnot groups are nilpotent, connected and simply connected, the exponential map is a diffeomorphism from $\mathfrak{
g}$ to $\mathbb{G}$, i.e. every $p \in \mathbb{G}$ can be uniquely written in the form
\begin{equation}p=\exp(p_1X_1+\dots+p_nX_n).\label{esp}\end{equation} By using this exponential coordinates, we can identify $p$ with the $n$-tuple $$(p_1,\dots,p_n)\in
\mathbb{R}^n$$ and identify $\mathbb{G}$ with $(\mathbb{R}^n,\cdot)$ where the explicit expression of the group operation $\cdot$ is determined by
Campbell-Hausdorff formula (see \cite{folland}). More precisely, $\mathbb G$ is isomorphic to the Lie group $(\mathbb{R}^n,\cdot)$.

\bigskip

\begin{definition}[Horizontal bundle]
A Carnot group is characterized by a horizontal bundle $H\mathbb{G}$, whose fibers are    \[ H\mathbb{G}_x=\{Y(x) \,|\, Y \in V_1\},\qquad x\in\mathbb{G} .\]
Moreover, the inner product $\langle \cdot,\cdot\rangle$ defined on $V_1$ induces an inner product $\langle \cdot,\cdot\rangle_x$ and a norm $|\cdot|_x$ on $H\mathbb G_x$ for every $x \in \mathbb G$; precisely,
\[ \scal{Y(x)}{Z(x)}{x}:=\langle Y,Z \rangle \; \textrm{ and } \; 
|Y(x)|_x^2:=\langle Y,Y\rangle.\]

The sections of $H\mathbb G$ are called horizontal
sections, a vector of $H\mathbb G_x$ is a horizontal vector.
Every horizontal section $\phi$ defined on an open set $\Omega \subset \mathbb G$ can be written as $\phi=\sum_{i=1}^{m_1} \phi_i X_i$, where its coordinates are functions $\phi_i: \Omega \rightarrow \mathbb R$. When considering two such sections
$\phi$ and $\psi$, we will write $\Scal
{\psi}{\phi}$ for $\scal {\psi(x)}{\phi(x)}{x}$; however the dependence on $x$ remains except for left invariant horizontal sections, i.e. vector fields belonging to $V_1$.
\end{definition}

\begin{definition}[Left translations and dilations]
 For each $x\in\mathbb G$, we define the left translation by $x$ as 
\begin{align*}
\tau_x:\,\mathbb G &\to\mathbb G \\
 z &\mapsto x\cdot z
\end{align*}
and, for each $\lambda >0$, we define the dilation $\delta_\lambda:\mathbb G\to\mathbb G$ as 
\begin{equation}\label{dilatazioni}
\delta_\lambda(x_1,...,x_n)=
(\lambda^{\alpha_1}x_1,...,\lambda^{\alpha_n}x_n),
\end{equation} where $\alpha_i\in\mathbb N$ is the homogeneity of variable $x_i$ in
$\mathbb G$  and it is defined as
\begin{equation}\label{omogeneita2}
\alpha_i=j \quad\text {when}\; 1+\sum_{k=1}^{j-1}m_k\leq i\leq \sum_{k=1}^j m_k \,.
\end{equation}

\end{definition}

Now we recall some definitions about metrics and measures  on Carnot groups. 
\begin{definition}[Sub-unit curve]
An absolutely continuous curve $\gamma:[0,T]\to \mathbb G$ is a sub-unit curve with respect to  $X_1,\dots,X_{m_1}$ if it is a horizontal curve, i.e. there exist measurable real functions $c_1(s),\dots,c_{m_1}(s)$, $s\in [0,T]$ such that
$$\dot\gamma(s)=\sum\limits_{j=1}^{m_1}\,c_j(s) X_j(\gamma(s)),
\quad \text{for a.e.}\;s\in [0,T],$$ and if $$\sum_jc_j^2\le 1.$$
\end{definition}

\begin{definition}[Carnot-Carath\'eodory distance] \label{distanza}
The Carnot-Carath\'eodory distance between two points $p$, $q\in\mathbb G$ is defined as \[
d_c(p,q):=\inf\left\{T>0:\;\text{there exists a sub-unit curve}\;
\gamma\;\text{with}\; \gamma(0)=p,\,\gamma(T)=q\right\}. \]
\end{definition}
The set of sub-unit curves joining $p$ and $q$ is not empty, since Chow's Theorem (see \cite{subriemann}, Theorem 1.6.2); furthermore  $d_c$
is a distance on $\mathbb G$ that induces the Euclidean topology (see \cite{ballmetrics}). We denote with $U_c(p,r)$ and $B_c(p,r)$
respectively the open and closed balls associated with $d_c$.
\bigskip

\begin{definition}[$d_\infty$ distance]
\cite{K1} defined another distance equivalent to the previous one.  $$d_{\infty}(x,y)=
d_{\infty}(y^{-1}\cdot x,0),$$ where, if $p=(\tilde p_1,\dots,\tilde p_k)\in \mathbb R^{m_1}\times\cdots\times\mathbb R^{m_k}=\mathbb R^n$, then
\begin{equation}\label{distanzainfinito}
d_{\infty}(p,0)=\max\{\varepsilon_j \vert\!\vert \tilde p_j \vert\!\vert_{\mathbb R^{m_j}}^{1/j}\, ,
j=1,\dots,k\}. \end{equation} Here $\varepsilon_1=1$, and
$\varepsilon_{2},\dots\varepsilon_{k} \in (0,1)$ are suitable positive constants depending on the group structure (see \cite{K1}, Theorem 5.1).
\end{definition}

\begin{definition}[Homogeneous dimension]
The integer
\begin{equation}\label{dimensione}
Q=\sum_{j=1}^n\alpha_j=\sum_{i=1}^k i\,\text{dim}\,V_i
\end{equation}
is the homogeneous dimension of $\mathbb G$. We stress that it is also the
Hausdorff dimension of $\mathbb R^n$ with respect to $d_c$ (see
\cite{Mi}).
\end{definition}

\begin{proposition}[Haar measure]
The $n$-dimensional Lebesgue measure  $\mathcal L^n$ is the Haar measure of the group $\mathbb G$ (see \cite{varsalcoul}). Therefore if $E\subset\mathbb R^n$ is measurable,
then $ \mathcal L^n(x\cdot E)= \mathcal L^n(E)$ for every $x\in\mathbb G$.
Moreover, if $\lambda>0$ then $\mathcal
L^n(\delta_\lambda(E))=\lambda^Q \mathcal L^n(E)$. We note that \begin{equation}\label{volumepalle} \mathcal
L^n(U_c(p,r))=r^Q \mathcal L^n(U_c(p,1))=r^Q \mathcal
L^n(U_c(0,1)).
\end{equation}
\end{proposition}
In our paper all the spaces $L^p(\mathbb G)$ are defined with respect to the Haar measure of $\mathbb G$.

\bigskip

Now we are ready to define the class of Carnot groups of type $\star.$
From now on we set $m:=m_1$.

\begin{definition}[Carnot groups of type $\star$]\label{mm}
We say that a Carnot group $\mathbb G$ is of type $\star$ if its stratified Lie algebra ${\mathfrak{g}}=V_1\oplus \dots \oplus V_\kappa$  has the following property: there exists a basis $(X_{1},\ldots,X_{m})$ of $V_{1}$
such that \begin{equation}[X_{j},[X_{j},X_{i}]]=0 \, \mbox{ for } \, i,j=1,\ldots,m \label{eq:norip}\end{equation}

In this case we also say that $\mathfrak{g}$ is of type $\star$.
\end{definition}

\begin{remark}
  It is clear that every step 2 Carnot group is of type $\star$, whereas free Carnot groups of step greater than 2 are not of type $\star$.
Moreover, if a Carnot group of step greater than 2 is of type $\star$, then the dimension of its first layer is at least 3; hence filiform groups of step greater than 2 (and in particular Engel group) are not of type $\star$. For details about the notions of free and filiform Carnot groups see \cite{lanco_bonfi_ugu}.

We do not require the validity of (\ref{eq:norip}) for every basis of $V_1$, since that would be equivalent to require that the step is 2. The proof is quite straightforward.
Suppose   $ [Y_{1},[Y_{1},Y_{2}]]=0$ for every $ Y_{1},Y_{2} \in V_1$. If $X,Y,Z \in V_1$ then    \begin{align*} 0&=[X+Y,[X+Y,Z]]=[X,[Y,Z]]+[Y,[X,Z]] \\  0&=[X+Z,[X+Z,Y]]=[X,[Z,Y]]+[Z,[X,Y]]=-2[X,[Y,Z]]+[Y,[X,Z]]. \end{align*}  Therefore $[X,[Y,Z]]=0$  for all $X,Y,Z \in V_1$. 
 \end{remark}
\begin{example}\label{unitupper}
The Lie group $\mathbb G_m$ of unit upper triangular $(m+1) \times (m+1)$ matrices  is a Carnot group of type $\star$, for any $m \in \mathbb N$. However, in order to avoid trivial cases, it is possible to assume $m>2$. This group is the nilpotent group that comes from the Iwasawa decomposition of $GL_{m+1}(\mathbb R)$.

\medskip

Now let us prove that $\mathbb G_m$ is of type $\star$ for $m>2$.
The Lie algebra $\mathfrak{g}_m$ of $\mathbb G_m$ is isomorphic to the one of strictly upper triangular $(m+1) \times (m+1)$ matrices (see \cite{unitupperbook}, Part I, Chapter 2, Section 5.7, Example 1).
If $E_{i,j}$ is the matrix with 1 in the $(i,j)$-th entry and 0 elsewhere, it is easy to see that a basis of $\mathfrak{g}_m$ is formed by the single-entry matrices $E_{k,k+l}$ for   $l=1,\ldots,m$ and $k=1,\ldots,m+1-l$, and $\mathrm{dim}\,\mathfrak{g}_m=\frac{m(m+1)}{2}$.
The choice of using the particular parameters $k$ and $l$ will soon be explained. 

The following formula, which can be proven by direct computation of the commutators of single-entry matrices, gives the expression of Lie brackets in $\mathfrak{g}_m$.   \begin{equation}\label{bracketslaw} [E_{k_1,k_1+l_1},E_{k_2,k_2+l_2}]=\begin{cases}E_{k_{1},k_1+(l_1+l_2)} & \mbox{if }k_{1}<k_{2} \mbox{ and } k_1+l_1 = k_2 \\
 -E_{k_{2},k_2+(l_1+l_2)} & \mbox{if } k_1>k_2 \mbox{ and } k_2+l_2=k_1\\
0 & \mbox{otherwise}. \\
\end{cases}
\end{equation}
From (\ref{bracketslaw}), it is easy to see that $E_{k,k+1}$ (for $k=1,\ldots,m$) are generators of $\mathfrak{g}_m$. 
Moreover, $\mathfrak{g}_m=V_1\oplus\cdots\oplus V_m$ with $$V_l=\mathrm{span}\{ E_{k,k+l} \; | \; k=1,\ldots,m+1-l \}$$ for $l=1,\ldots,m$. This explains the use of the parameters $k$ and $l$.
Moreover, we observe that $m$ is the dimension of $V_1$ and the step of the stratification.

Now we can finally prove that $\mathbb G_m$ is of type $\star$.
We set $X_k:=E_{k,k+1}$ for $k=1,\ldots,m$.
From (\ref{bracketslaw}) we obtain that $E_{k,k+2}=[X_k,X_{k+1}]$ for $k=1,\ldots,m-1$ and the other independent commutators of length 2 are zero, whereas $E_{k,k+3}=[[X_k,X_{k+1}],X_{k+2}]=[X_k,[X_{k+1},X_{k+2}]]$ for $k=1,\ldots,m-2$ and the other independent commutators of length $3$ are zero. Hence (\ref{eq:norip}) holds.

\medskip

The identity (\ref{bracketslaw}) allows to explicitly write an adapted basis of $\mathfrak{g}_m$ and the expression of Lie brackets for any $m \in \mathbb N$, but we also want to remark the trivial case obtained when $m=2$, that is the Lie algebra of the three-dimensional Heisenberg group.
\end{example}
\begin{example}
Another example of stratified algebra of type $\star$ (besides $\mathfrak{g}_m$) is obtained from
$$\mathfrak{g}_3=\mathrm{span}\{X_1,X_2,X_3 \}\oplus\mathrm{span}\{[X_1,X_2],[X_2,X_3] \}\oplus\mathrm{span}\{[[X_1,X_2],X_3]\},$$
where $ [X_1,X_3]=0, \, [[X_1,X_2],X_3]=[X_1,[X_2,X_3]]$ and the other commutators of length $3$ are zero, by setting $[X_1,X_3]=b[X_2,X_3]$ with $ b \neq 0$.
This can be done, since Jacobi identity 
$$[[X_1,X_2],X_3]+[[X_2,X_3],X_1]+[[X_3,X_1],X_2]=0$$
is still verified.
The only non-zero commutators of length 3 are
$$[[X_1,X_2],X_3]\!=\![X_1,[X_2,X_3]]\!=\![[X_3,X_2],X_1] \textrm{ and } [[X_3,X_1],X_1]\!=\!b[[X_1,X_2],X_3].$$ 
 By changing the basis of the first layer in the following way
 \begin{displaymath}
\begin{cases}\tilde X_1=X_1-bX_2  \\
 \tilde X_2=X_2 \\
 \tilde X_3=X_3 
\end{cases}
\end{displaymath}
it is easy to see that $\mathfrak{g}$ is of type $\star$.
\end{example}
\begin{definition}
An ideal of a graded Lie algebra is said to be homogeneous if can be generated by homogeneous elements.
\end{definition}
 Obviously, the quotient of a stratified Lie algebra by an homogeneous ideal is still a stratified Lie algebra. 
\begin{remark}
A stratified Lie algebra of type $\star$ can contain filiform stratified subalgebras of step greater than 2. 

For instance, let us consider the free Lie algebra of step 3 with 3 generators. Now we quotient it by the homogeneous ideal generated by $[X_{j},[X_{j},X_{i}]]$ with $i,j=1,2,3$. Obviously the obtained stratified Lie algebra is of type $\star$. If we consider its stratified subalgebra
 $\text{Lie}\{X_1+X_2,X_3   \} $, we can verify it is filiform of step 3.
 \begin{align*}
 [X_1+X_2,X_3]&=[X_1,X_3]+[X_2,X_3]\neq 0 \\ [X_1+X_2,[X_1+X_2,X_3]]&=[X_2,[X_1,X_3]]+[X_1,[X_2,X_3]]\neq 0\\  [X_3,[X_1+X_2,X_3]]&=0.
 \end{align*}
 Thus, type $\star$ property is not inherited by stratified subalgebras, but is obviously inherited by stratified quotient algebras.       
\end{remark}

\begin{remark}\label{dimensionreason}
In a stratified algebra, $\mathrm{dim}\,V_3 \leq \frac{1}{3}(m+1)m(m-1)$ and $\mathrm{dim}\,V_3 = \frac{1}{3}(m+1)m(m-1)$ in free stratified algebras of step greater than 2.
In fact, there are $m(m-1)$ commutators of form $[X_j,[X_j,X_i]]$ with $j \neq i$, $\binom{m}{3}$ commutators of form $[X_i,[X_j,X_k]]$ with $i<j<k$ and
$\binom{m}{3}$ commutators of form $[X_k,[X_j,X_i]]$ with $i<j<k$, which span $V_3$ and are linearly independent, if we consider only  the relations of antisymmetry and Jacobi identities. If we sum the number of these commutators, we get $$m(m-1)+2 \binom{m}{3}=\frac{1}{3}(m+1)m(m-1).$$

On the contrary, in a stratified Lie algebra of type $\star$, $\mathrm{dim}\,V_3 \leq \frac{1}{3}m(m-1)(m-2)$.

\medskip

For a dimension reason we can say again that free stratified algebras of step greater than 2 are not of type $\star$ and that stratified algebras of type $\star$ with $m=2$ cannot be of step greater than $2$, hence filiform algebras of step greater than 2 are not of type $\star$.
\end{remark}

\section{Sets with constant intrinsic normal} \label{constantnormal}
In this section we define the intrinsic perimeter measure, the intrinsic normal and we show the main result of our paper: Theorem \ref{constant_halfspace}.

\begin{definition}[Sets of locally finite perimeter]
A measurable set $E \subset \mathbb G$ has locally finite perimeter (or is a $\mathbb G$-Caccioppoli set) if $X \mathbbm{1}_E$ is a Radon measure on $\mathbb G$ for any $X \in V_1$.
\end{definition}

\begin{definition}[Perimeter measure and generalized normal]
If $f \in L^1_\text{loc}(\mathbb G)$ with $X_i f$ Radon measures for $1 \leq i \leq m$, it is possible to define the horizontal gradient of $f$ as the $\mathbb R^m$-valued Radon measure
\[  \nabla_\mathbb G f := (X_1 f, \ldots , X_m f ) . \]
If $E$ is a $\mathbb G$-Caccioppoli set, the total variation $|\nabla_\mathbb G \mathbbm{1}_E|$ of $\nabla_\mathbb G \mathbbm{1}_E$ is the perimeter measure of $E$ in $\mathbb G$ and we will denote it with $|\partial E|_\mathbb G$.
Moreover there exists a $|\partial E|_\mathbb G$-measurable horizontal section $\nu_E$ on $\mathbb G$, such that $|\nu_E(x)|_x=1$ for $|\partial E|_\mathbb G\text{-a.e. } x \in \mathbb G$ and
\begin{equation*}
\nabla_\mathbb G \mathbbm{1}_E=\nu_E \res |\partial E|_\mathbb G=((\nu_E)_1 \res |\partial E|_\mathbb G,\ldots, (\nu_E)_m \res |\partial E|_\mathbb G)
\end{equation*}
where $ \nu_E =\sum_{i=1}^m (\nu_E)_i X_i $.
We say that $\nu_E$ is the generalized intrinsic normal of $E$.
\end{definition}

\begin{remark}
We stress that $\nu_E$ is an inward-pointing normal and is defined up to a $|\partial E|_\mathbb G$-negligible set, therefore we assume that $E$ has a non-null perimeter. In particular the perimeter measure is concentrated in a subset of the topological boundary of $E$, the so-called reduced boundary $\partial^\ast_\mathbb G E$.
\end{remark}

\begin{definition}[Reduced boundary]\label{reduced_boundary}
Let $E$ be a $\mathbb G$-Caccioppoli set; we say that $x\in\partial^{*}_{\mathbb G} E$ if
\begin{gather}
\per{E}{U_c(x,r)} >0 \qquad \text {for any}\; r>0; \tag{$i$}\\
\text{there exists}\quad \lim_{r\to
0}\ave_{U_c(x,r)}\nu_E\;d\Per{E};\tag{$ii$}\\
{\left\Vert {\lim_{r\to 0}\ave_{U_c(x,r)}\nu_E \;d\Per{E}}
\right\Vert}_{\mathbb R^{m_1}}=1.\tag{$iii$}
\end{gather}
\end{definition}

\begin{lemma}[Differentiation Lemma \cite{K1}]\label{difflemma}
Assume $E$ is a $\mathbb G$-Caccioppoli set, then
\begin{equation*}
\lim_{r \to 0}\ave_{U_c(x,r)} \nu_E\;d\Per{E}=\nu_E(x),\qquad
\text{for}\;\Per{E}\text{-a.e.}\; x,
\end{equation*}
hence $\Per{E}$ is concentrated on the reduced boundary
$\partial^{*}_{\mathbb G} E$.
\end{lemma}

\begin{remark} Thanks to  Lemma
\ref{difflemma}, we can redefine $\nu_E$ in a $\Per{E}$-negligible set, by setting  $\nu_E(x)=\lim_{r \to 0}\nabla_\mathbb G \mathbbm{1}_E(U_c(x,r))/\Per{E}(U_c(x,r))$  at every
point $x\in \partial^{*}_{\mathbb G} E$.
\end{remark}

\begin{remark}
Now we want to consider the case in which $\nu_E$ is a constant. 
From the previous definition it is clear that
\begin{equation*}   X_i \mathbbm{1}_E = (\nu_E)_i \res |\partial E|   
\; \text{ for }  1 \leq i \leq m.    \end{equation*}
Hence we can observe that 
 $\nu_E=X_1$  if and only if $X_1 \mathbbm{1}_E \geq 0$ and $X_i \mathbbm{1}_E=0$ for $2 \leq i \leq m$.   
 Thus the following proposition is justified.
\end{remark}

\begin{proposition}[Constant intrinsic normal set]
A set of locally finite perimeter $E \subset \mathbb G$ has a constant intrinsic normal if and only if there exists $X \in V_1$ such that $X \mathbbm{1}_E \geq 0$, $X \mathbbm{1}_E \not\equiv 0$ and
$Y \mathbbm 1_E=0$ for every $Y$ orthogonal to $X$ in $V_1$.
\end{proposition}

\begin{definition}[Vertical halfspace]
A set $H \subset \mathbb G$ is a vertical halfspace if it has a constant intrinsic normal and \[ Y \mathbbm 1_H=0 \quad \forall \, Y \in \bigoplus_{i=2}^\kappa V_i. \] 
A vertical halfspace can be represented as 
\[ \tau_q\left( \left\{x  :  \sum_{i=1}^m v_i x_i \geq 0  \right\} \right) \]
for some $q \in \mathbb G$ and some unit vector $v \in \mathbb R^m$. 
\end{definition}

\begin{theorem}\label{constant_halfspace}
Let $\mathbb G$ be a group of type $\star$.
Then every set with a constant intrinsic normal is a vertical halfspace.
\end{theorem}

In order to prove Theorem \ref{constant_halfspace} we need Lemma \ref{lemma_norip} below, which shows an important property of stratified Lie algebras of type $\star$. 
\begin{lemma}\label{lemma_norip}
Let $\mathfrak{g}=V_1\oplus\cdots\oplus V_\kappa$ be a stratified Lie algebra of type $\star$. Then for every basis $(Y_1,\ldots,Y_m)$ of $V_1$
\begin{equation}\nonumber
[Y_1,[Y_1,Y_p]]=\sum_{j > 1} \alpha_{pij}[Y_j,[Y_j,Y_i]]+\sum_{k\neq j,k\neq i} \beta_{pijk} [Y_k,[Y_j,Y_i]]
\end{equation}
 holds for $p=2,\ldots,m$ (with $\alpha_{pij},\beta_{pijk} \in \mathbb{R} $).

\medskip

\textnormal{We stress that the first sum contains commutators ``with repeated indices different from 1'', whereas the second one contains commutators ``without repeated indices''. Roughly speaking, the previous lemma states that a commutator where an index $i$ is repeated (for instance $i=1$) can be written as a linear combination of the remaining commutators excluding those where the  index $i$
is repeated.}
\end{lemma}
\begin{proof}
Let $(Y_1,Y_2,\ldots,Y_m)$ be any basis of $V_1$. Obviously $$V_3=\mathrm{span}\{[Y_k,[Y_j,Y_i]] \;|\;i,j,k=1,\ldots,m  \} .$$
Now let $(X_1,X_2,\ldots,X_m)$ be a basis of $V_1$ that respects (\ref{eq:norip}).
The relation between the two basis is
\begin{equation}\label{changebasis} 
\left(\!\begin{array}{c}
X_1 \\
X_2 \\
\cdots \\
X_m \\
\end{array}\!\right) =A \left(\! \begin{array}{c}
Y_1 \\
Y_2 \\
\cdots \\
Y_m \\
\end{array} \!\right) \end{equation}
with
$$
A= \left(\!\begin{array}{cccc}
a_{11} & a_{12} & \cdots & a_{1m} \\
a_{21} & a_{22} & \cdots & a_{2m} \\
\cdots & \cdots & \cdots & \cdots \\
a_{m1} & a_{m2} & \cdots & a_{mm} \\
\end{array}\!\right)
$$
invertible real $m \times m $ matrix.
Hence there exists $ i \in \{1,\ldots,m  \}$ such that $a_{i1} \neq 0 $. 
After reordering the basis $(X_1,X_2,\ldots,X_m) $, we can assume $a_{11}\neq 0$.
Now we recall the relations $$[X_1,[X_1,X_h]]=0 \quad \mbox{ for } h=2,\ldots,m .$$  
Replacing  (\ref{changebasis}) in these identities,
 we get $$\sum_{i,j,k} a_{1k}a_{1j}a_{hi} [Y_k,[Y_j,Y_i]]=0 \quad \mbox{ for } h=2,\ldots,m.  $$
Keeping in mind the antisymmetry of brackets, we obtain
 $$\sum_{i=2}^m (a_{11}a_{11}a_{hi}-a_{11}a_{h1}a_{1i}) [Y_1,[Y_1,Y_i]]+\sum_k \alpha_{hk}Z_k=0 \quad \mbox{ for } h=2,\ldots,m $$ where $Z_k$ are all the remaining commutators of length 3, i.e.  commutators 
 that are not of the form $[Y_1,[Y_i,Y_1]]$ with $i=2,\ldots,m$ which are all incorporated in the first sum.  
Now we move all terms containing $Z_k$ to the right-hand side and we get
$$ \sum_{i=2}^m (a_{11}a_{hi}-a_{h1}a_{1i}) [Y_1,[Y_1,Y_i]]=-\frac{1}{a_{11}}\sum_k \alpha_{hk}Z_k  \mbox{ for } h=2,\ldots,m.$$
%
If we denote by $U_\ell$ the $\ell$-th component of a vector $U \in V_3$, $\ell=1,\dots,\dim V_3$
with respect to an arbitrarily fixed basis,
then the components $[Y_1,[Y_1,Y_i]]_\ell$ of $[Y_1,[Y_1,Y_i]]$
 are solutions of a $(m-1) \times (m-1) $ linear system
with coefficient matrix  
$$M=(m_{ij})_{i,j=1,\ldots,m-1} \quad \mbox{ where } m_{ij}=a_{1,1}a_{i+1,j+1}-a_{i+1,1}a_{1,j+1}. $$
We stress that $ m_{ij}$ are the second order minors of $A$ that contains $a_{11}$. $M$ is invertible since  
$$ \det M=(a_{11})^{m-2} \det A \neq 0 .$$
This computation comes from the so-called \textit{Chio's pivotal condensation} (see, e.g., \cite{K2}, Theorem 3.6.1).
Therefore, we obtain eventually
\begin{align}\nonumber 
 \left(\!\begin{array}{c}
([Y_1,[Y_1,Y_2]])_\ell\\
\cdots  \\
([Y_1,[Y_1,Y_m]])_\ell\\
\end{array}\!\right) &=-\frac{1}{a_{11}}\sum_k (Z_k)_\ell \;  M^{-1}\left(\!\begin{array}{c}
\alpha_{2k} \\
\cdots  \\
\alpha_{mk} \\
\end{array}\!\right)    &\mbox{ for } \ell =1,\ldots,\dim V_3.
\end{align}
\end{proof}

\begin{proof}[Proof of Theorem \ref{constant_halfspace}]
Since $H$ has constant intrinsic normal, there exist an orthonormal basis ($Y_1$,$Y_2$, $\ldots$ ,$Y_m$) for $V_1$ such that \begin{equation}\label{ipo}Y_1 \mathbbm 1_H \geq 0 \quad \textrm{and} \quad Y_j \mathbbm 1_H=0 \; \textrm{ for} \; 2 \leq j \leq m .\end{equation} In order to prove that $H$ is a vertical halfspace, it is sufficient to show that \begin{equation}\label{tesi}Z \mathbbm 1_H=0 \quad \forall \; Z \in \bigcup_{j=2}^k V_j .\end{equation}       
The function $\mathbbm 1_H$ can be approximated by smooth functions using the group convolution (see \cite{FS}). Therefore, without loss of generality, we can replace $\mathbbm 1_H$ with a smooth function $g$ in (\ref{ipo}) and (\ref{tesi}).

  If $(X_1,\ldots,X_m)$ is a basis for $V_1$ satisfying (\ref{eq:norip}), then
\begin{equation}\label{sist_prec} Y_{1}=\sum a_{1i}X_{i}\,,\; \ldots \: ,\; Y_{m}=\sum a_{mi}X_{i} .\end{equation} We denote with $A$ the coefficient matrix \begin{equation}
A=\left(\begin{array}{cccc}
a_{11} & a_{12} & \ldots & a_{1m}\\
a_{21} & a_{22} & \ldots & a_{2m}\\
\ldots & \ldots & \ldots & \ldots\\
a_{m1} & a_{m2} & \ldots & a_{mm}\end{array}\right),\qquad\det A\neq0.\label{eq:nonsing}\end{equation}
Now we notice that (\ref{sist_prec}) yields
$$
\left(\!\!\begin{array}{c}
Y_1 g(x)\\
\ldots\\
Y_m g(x)\end{array}\!\!\right)=
A\left(\!\!\begin{array}{c}
X_1 g(x)\\
\ldots\\
X_m g(x) \end{array}\!\!\right)
\qquad \mbox{for every } x \in \mathbb G.
$$
From (\ref{eq:nonsing}) $A$ is invertible, hence the inverse image of $\{ (t,0,\ldots,0) \in \mathbb{R}^m \, | \, t \in \mathbb{R}    \}$ under $A$ is a line of $\mathbb{R}^m$, i.e. there exists $k \in \mathbb{R}^m$,  $k \neq 0$ such that
\begin{equation}\label{inverseimage}   A^{-1}\{ (t,0,\ldots,0) \in \mathbb{R}^m \, | \, t \in \mathbb{R}    \}= \{  \lambda k \, | \, \lambda \in \mathbb{R}  \}.     \end{equation} 
Since $k=(k_1,\ldots,k_m) \neq 0$, at least one of its components is not zero: for instance we can assume $k_m \neq 0$ and then, without loss of generality,  $k_m=1$.
We stress that $k$ depends only on $A$, so that the vector $k$ is independent of $x\in\mathbb G$.

If we indicate with $A_j$ the $j$-th row of $A$, we have \begin{equation}\label{wow_nozero} \langle A_1 , k \rangle \neq 0 \: .\end{equation} Indeed by the definition of $k$, $\langle A_j , k \rangle = 0$ for $j=2,\ldots,m$ and $\langle A_1 , k \rangle\neq 0$ by the invertibility of $A$.

Remember now that, by (\ref{ipo}), $Y_{1}g\geq0,\; Y_{2}g=0,\;\ldots ,\; Y_{m}g=0$. Therefore, if $x\in\mathbb G$, then
by (\ref{inverseimage}), $X_i g(x)=\lambda(x) k_i$ for $i=1,\ldots,m$. In particular, taking $i=m$,
we have $\lambda(x)=X_m g(x)$, which leads to $X_i g(x)=k_i X_m g(x) $ for $i=1,\ldots,m$.

Now $Y_1 g(x)= \langle A_1 ,  k  \rangle X_m g(x)$. We recall that $Y_1 g(x) \geq 0$ and (\ref{wow_nozero}), so that we can conclude that $X_m g(x) \geq 0$ for every $x \in \mathbb G$ or $X_m g(x) \leq 0$ for every $x \in \mathbb G$. 
\bigskip

Now let us show that $Zg=0$ for every $Z \in V_2$. 

\bigskip

Let $Z\in V_{2}$. Then $$Z=\sum_{i_{1},i_{2} \in \{1,\ldots,m \}}\alpha_{i_{1},i_{2}}[X_{i_{1}},X_{i_{2}}] \qquad (\alpha_{i_{1},i_{2}} \in \mathbb{R})$$
where $X_{i_{1}}g=k_{i_1}X_{m}g$, $\,X_{i_{2}}g=k_{i_2}X_{m}g$ and $X_{m}g\geq0$
(or $X_{m}g\leq0$).

The case $[X_{i_{1}},X_{i_{2}}]g=0$ is trivial. In particular, we can assume that,
 for instance $k_{i_1}\neq 0$. Then
we have $X_{m}g=(k_{i_1})^{-1}X_{i_{1}}g$ from which \begin{equation}\label{heis}\mbox{$X_{i_{2}}g=k_{i_2}(k_{i_1})^{-1}X_{i_{1}}g \qquad$   with $X_{i_{1}}g\geq0$ or $X_{i_{1}}g\leq0$.}
\end{equation}

On the other hand, by hypothesis $[X_{j},[X_{j},X_{i}]]=0$  (with $i,j=1,\ldots,m$),
so that
$$\textrm{span}\{X_{i_{1}},X_{i_{2}}-k_{i_2}(k_{i_1})^{-1}X_{i_{1}},[X_{i_{1}},X_{i_{2}}]\}$$
 is a Lie algebra isomorphic to the Lie algebra
 of the Heisenberg group $\mathbb{H}^1$ (remember $[X_{i_{1}},X_{i_{2}}]\neq 0$). Indicate with $\mathfrak{h}$ the Lie algebra of $\mathbb{H}^1$.
The following claim, shown inside the proof of Lemma 3.6 in \cite{K1}, holds:
\begin{itemize}
\item if $\tilde{X}_1$,$\tilde{X}_2 \in \mathfrak{g}$ and $\tilde{X}_1 g \geq 0$, $\tilde{X}_2 g = 0$ and
   $\tilde{\mathfrak{g}}:=\textrm{span}\{\tilde{X}_1,\tilde{X}_2,[\tilde{X}_1,\tilde{X}_2]\}$  is a subalgebra of $\mathfrak{g}$ isomorphic to $\mathfrak{h}$,
 then $[\tilde{X}_1,\tilde{X}_2]g=0$ .
\end{itemize}
Alternatively this claim can be seen as an easy consequence of Remark 4.9 of \cite{newambro}. 
Recalling (\ref{heis}), we can conclude that $[X_{i_{1}},X_{i_{2}}]g=0$. Thus $Zg=0$ for every $Z \in V_2$. 

\bigskip

Now, in order to deal with vector fields belonging to $V_3$, we use the basis $(Y_{1},\ldots,Y_{m})$ of $V_1$. Since it is a basis, for every $W\in V_{3}$ there exist $ Z_j \in V_{2}$ such that $ W=\sum_{i,j}\beta_{ij}[Y_{i},Z_j]$. First of all, 
\begin{equation}\label{i>1}
\mbox{
if $i=2,\ldots,m$ then $[Y_{i},Z_j]g=0$, since $Z_jg=0$ e $Y_{i}g=0$.}
\end{equation}
Consider now the case $i=1$. We notice that $Z_j$ can be written as a linear combination of two types of commutators: 
$[Y_{1},Y_{k}]$ and $[Y_{l},Y_{i}]$, with $l \neq 1$ and $i \neq 1 $.
Hence, dropping the index $j$ from $Z_j$, we get
$$[Y_{1},Z]=\sum_{k}\gamma_{k}[Y_{1},[Y_{1},Y_{k}]]+\sum_{ l\neq1,\, i\neq1}\lambda_{li}[Y_{1},[Y_{l},Y_{i}]].$$
Moreover, by Lemma \ref{lemma_norip}, $$[Y_{1},[Y_{1},Y_{k}]]=\sum_{j \neq 1}\mu_{kij}[Y_{j},[Y_{j},Y_{i}]]+\sum_{l \neq j, l \neq i}\theta_{kijl}[Y_{l},[Y_{j},Y_{i}]] . $$
Thus, $[Y_{1},Z]g$ can be written as a linear combination of three types of functions:
\begin{itemize}
\item[i)] $[Y_{j},Z]g$ with $j \neq 1$ and $Z \in V_2$;
\item[ii)] $[Y_{l},[Y_{j},Y_{i}]]g$ with $l \neq j, l \neq i$;
\item[iii)] $[Y_{1},[Y_{l},Y_{i}]]g$ with $l\neq1,\, i\neq1$.
\end{itemize}
By \eqref{i>1} above, terms of type i) vanish. Analogously,
terms of type iii) vanish since they can be reduced, by Jacobi identity, to a sum of terms of type i).
Finally, as for terms of type ii), either $l>1$ or $l=1$. If $l>1$, they vanish again by \eqref{i>1}. If $l=1$, then necessarily
$j,i\neq 1$ and then, again by Jacobi identity, they can be written as a sum of two terms of type i).
%

Therefore, $Wg=0$ for every $W \in V_{3}$.
\bigskip

In order to complete the proof for the other layers, we show that for every $k\geq3$ the following claim holds:
\begin{align}
&\textrm{for every $W\in V_{k}$ there exist $Z_j \in V_{k-1}$, $\tilde{Z}_{r} \in V_h$ and } \nonumber \\
&\textrm{$\hat{Z}_{s} \in V_{k-h}$ with $2 \! \leq \! h \! \leq \! k \! - \! 2$ such that $W$ is a linear combination }
 \label{asterisco} \\ \nonumber &\textrm{of commutators of the form $[Y_{l},Z_j]$
with $l>1$ and $[\tilde{Z}_{r},\hat{Z}_{s}]$.}
\end{align}
We argue by induction on $k$.
We have just seen the case $k=3$. Suppose now $n\geq4$, and
assume the claim is true for $k\leq n-1$. We show it holds for
$k=n$.

Indeed, for every $W\in V_{n}$ there exist $Z_i \in V_{n-1}$ such that $W$ is a linear combination of commutators $[Y_{j},Z_i]$
with $j=1,\ldots,m$. Obviously, the only commutators to work on in order to show (\ref{asterisco}) are those  with $j=1$.

But every $Z_i \in V_{n-1}$ is a linear combination of commutators of type $[Y_{l},\tilde{Z}]\;$ \begin{small}$(\textrm{with}\;\tilde{Z}\in V_{n-2}\; \textrm{and} \; l>1)$\end{small}
and
$[Z_{1},Z_{2}]\:$ \begin{small}$(\mbox{with }Z_{1}\in V_{h},\; Z_{2}\in V_{n-1-h},\;2\leq h\leq n-3)$\end{small},
we can reduce $[Y_{1},Z_i]$ to a linear combination of commutators of type \[ [Y_{1},[Y_{l},\tilde{Z}]]=-[\tilde{Z},[Y_{1},Y_{l}]]-[Y_{l},[\tilde{Z},Y_{1}]] \]
and
\[
[Y_{1},[Z_{1},Z_{2}]]=-[Z_{1},[Z_{2},Y_{1}]]-[Z_{2},[Y_{1},Z_{1}]].
\]
Hence (\ref{asterisco}) holds for every $k \geq 3$.

Consequently, $Z g=0$ for every $Z \in V_k$, with $k \geq 3$.
Therefore the theorem is proved.
\end{proof}

\section{Rectifiability} \label{geometricmotivation}
As a consequence of Theorem \ref{constant_halfspace} it is possible to extend Rectifiability theorem to Carnot groups of type $\star$. First we will remind some preliminary notions.  
\begin{definition}[$\mathbb G$-linearity]
A map $L:\;\mathbb G \to \mathbb R$ is $\mathbb G$-linear if it is a homomorphism
from $\mathbb G\equiv(\mathbb R^n,\cdot)$ to $(\mathbb R,+)$ and if it is positively
homogeneous of degree $1$ with respect to the dilations of $\mathbb G$,
that is $L(\delta_\lambda x)= \lambda Lx$ for $\lambda >0$ and
$x\in\mathbb G$. The $\mathbb R$-linear set of $\mathbb G$-linear functionals
$\mathbb G\to\mathbb R$ is indicated as $\mathcal L_\mathbb G$ and it is endowed with
the norm $$ \Vert L\Vert_{\mathcal
L_\mathbb G}:=\sup\{|L(p)|:\;d_{c}(p,0)\leq1\}. $$

\end{definition}

\begin{definition}[Pansu diffentiability] \label{differential}
Let $\Omega$ be an open set in $\mathbb G$, then $f:\Omega \rightarrow
\mathbb R$ is Pansu differentiable (see \cite{pansu} and \cite{koranireim}) in $x_0$ if there exists a $\mathbb G$-linear map $L$ such that $$ \lim_{x\to
x_0}\frac{f(x)-f(x_0)-L(x_0^{-1}\cdot x)}{ d_c(x,x_0)}= 0. $$
An equivalent definition is the following one:
there exists a group homomorphism $L$ from $\mathbb G$ to $(\mathbb R,+)$ such that $$
\lim_{\lambda\to 0+} \frac{f(\tau_{x_0}(\delta_{\lambda}
v))-f(x_0)}{ \lambda} = L(v) $$ uniformly with respect to $v$
belonging to compact sets in $\mathbb G$. 

In particular, $L$ is unique
and we write $L=d_{\mathbb G}f(x_0)$.  We remark that Pansu differential depends only on $\mathbb G$ and not on a particular choice of a basis of $\mathfrak g$.
\end{definition}

\begin{definition}[$\mathbf C^1_\mathbb G$ functions and sections]\label{C1H}
If $\Omega$ is an open set in $\mathbb G$, we denote with $\mathbf C^1_{\mathbb G} (\Omega)$
the set of continuous real functions in  $\Omega$ such that
$d_{\mathbb G}f:\Omega\to\mathcal L_\mathbb G$ is continuous in $\Omega$.
Furthermore, we denote with $\mathbf C^1_{\mathbb G} (\Omega,H\mathbb G)$ the set of all
sections $\phi=\sum_{i=1}^{m_1} \phi_i X_i$ of $H\mathbb G$ whose coordinates $\phi_i\in\mathbf C^1_{\mathbb G}
(\Omega)$ for $i=1,\dots,m_1$.
\end{definition}

\begin{definition}[Horizontal divergence]
If
$\phi=(\phi_1,\dots,\phi_{m_1})$ is a horizontal section such
that $X_j\phi_j\in L^1_{\rm loc}(\mathbb G)$ for $j=1,\dots,m_1$,  the horizontal divergence of $\phi$ is the real valued function
\begin{equation*}
\mathrm{div}_{\mathbb G}\,(\phi):=\sum_{j=1}^{m_1}X_j\phi_j.
\end{equation*}
\end{definition}

\medskip

If ${\Omega}\subset \mathbb G$ is open, the space of compactly supported smooth
sections of $H\mathbb G$ is denoted by ${\mathbf
C}_0^{\infty}({\Omega},H\mathbb G)$. If $k\in\mathbb N$, ${\mathbf
C}_0^{k}({\Omega},H\mathbb G)$ is defined similarly.

\medskip

For details and proofs about Calculus of Variations on Carnot groups, see \cite{garnie} and \cite{FSSC1}.

\medskip

\begin{definition}[$\mathbb G$-regular hypersurface]\label{Gregolare} $S\subset \mathbb G$ is a 
$\mathbb G$-regular hypersurface if for every
$x\in S$ there exist a neighborhood $\mathcal{U}$ of $x$ and a
function $f\in\mathbf C^1_{\mathbb G} (\mathcal{U})$ such that
\begin{gather} S\cap \mathcal{U} = \{y\in \mathcal{U}:
f(y)=0\};\tag{$i$}\\ \nabla_{\mathbb G}f(y)\neq 0\tag{$ii$}\quad\mbox{
for }y\in \mathcal{U}.
\end{gather}
\end{definition}

\begin{definition}[$\mathbb G$-rectifiable set]\label{def rectifiability}
$\Gamma\subset \mathbb G$ is said to be $(Q-1)$-dimensional
$\mathbb G$-rectifiable if there exists a sequence of $\mathbb G$-regular
hypersurfaces $(S_j)_{j\in\mathbb N}$ such that
\begin{equation}\label{rect}
\mathcal H_{c}^{Q-1}\left(\Gamma\setminus\bigcup_{j\in\mathbb
N}S_j\right) = 0,
\end{equation}
where $\mathcal H_{c}^{Q-1} $ is the ($Q-1$)-dimensional Hausdorff measure related to the distance $d_c$.
\end{definition}

\begin{definition}[Tangent group and tangent plane]
If $S=\{x:f(x)=0\}\subset \mathbb G$ is a $\mathbb G$-regular hypersurface, the
 tangent group to $S$ at $x$ is
the proper subgroup of $\mathbb G$ defined as$$T_{\mathbb G}^gS(x):=\{v\in\mathbb G :\;{\langle\nabla_\mathbb G f(x),\pi_xv
\rangle}_x=0 \},$$ where $\pi_xv:=\sum_{j=1}^{m_1} v_j X_j(x) $.   

We can also define the tangent plane to
$S$ at $x$ as $$ T_{\mathbb G}S(x):=x\cdot T_{\mathbb G}^gS(x). $$
\end{definition}
This definition is good; in fact the tangent plane does not depend on
the particular function $f$ defining the surface $S$ because of
point $(iii)$ of  Theorem 2.1 (Implicit Function Theorem) in \cite{FSSC dini} that yields $$
T_\mathbb G^gS(x)=\{v\in\mathbb G:\;{\langle \nu_E(x),\pi_xv \rangle}_x=0\} $$
where $\nu_E$ is the intrinsic normal.

We stress that the notion of $\mathbb G$-regular hypersurfaces is different
from the one of Euclidean $\mathbf C^1$-hypersurfaces in $\mathbb R^n$. In particular, in Corollary \ref{Balog}  we will consider
Euclidean $\mathbf C^1$-surfaces, which can have characteristic points, i.e. points $p\in S$ where the Euclidean tangent plane $T_pS$
contains the horizontal fiber $H\mathbb G_p$.
If $S$ is an Euclidean $\mathbf C^1$-hypersurface in $\mathbb G$, we denote with $\mathcal C(S)$ the set of its characteristic points.
The tangent group does not exist
in these points;
however, there is an important result about them proved in \cite{magnani3}:
in any Carnot group it holds that, if $S$ is a  $\mathbf C^1$-hypersurface, $\mathcal H^{Q-1}_c(\mathcal C(S))=0$.

\bigskip

In \cite{K1}, the rectifiability theorem is proved for step 2 Carnot groups and Blow-up Theorem is the main key of the proof and also the reason of the restriction to step 2.
In fact, there is a counterexample regarding a particular step 3 Carnot group, i.e. the Engel group, for which Blow-up Theorem does not hold (see \cite{K1}, Example 3.2).
In particular, in the Engel group there exist cones (i.e. dilation-invariant sets) with constant horizontal normal that are not vertical halfspaces, but nevertheless have the vertex belonging to the reduced boundary.

 The problem of rectifiability in general Carnot groups
remains an open question.
 Here we recall the Blow-up Theorem.
\bigskip

Let $\mathbb G$ be a  Carnot group.
For any set $E\subset \mathbb G$, $x_0\in\mathbb G$ and $r>0$ we
define the sets
$$ E_{r,x_0}
:=\{x: x_0\cdot \delta_r(x) \in E\}=\delta_{\frac1
r}\tau_{x_0^{-1}}E.$$
If $v\in H\mathbb G_{x_0}$ we define the
halfspaces $S^+_{\mathbb G}(v)$ and $S^-_{\mathbb G}(v)$ as
\begin{equation}\label{halfspace1}
\begin{split}
S^+_{\mathbb G}(v)&:=\{x: \scal{ \pi_{x_0}x}{v}{x_0}\geq 0\}\\
S^-_{\mathbb G}(v)&:=\{x: \scal{ \pi_{x_0}x}{v}{x_0}\leq 0\}.
\end{split}
\end{equation}
The common topological boundary $T^g_{\mathbb G}(v)$ of $S^+_{\mathbb G}(v)$ and
of $S^-_{\mathbb G}(v)$ is the subgroup of $\mathbb G$
\begin{equation}
T^g_{\mathbb G}(v):=\{x: \scal{ \pi_{x_0}x}{v}{x_0}= 0\}.\nonumber
\end{equation}
Moreover, we shall denote with $\mathcal
H^{n-1}$ the $(n-1)$-dimensional Hausdorff measure related to the
Euclidean distance in $\mathbb R^n\simeq \mathbb G$, with $\mathcal S^{Q-1}_{c}$ the
$(Q-1)$-dimensional spherical Hausdorff measure related to the distance $d_c$
in $\mathbb G$, and with $\mathcal S^{Q-1}_{\infty}$ the $(Q-1)$-dimensional spherical
Hausdorff measure related to the distance $d_\infty$ in $\mathbb G$.

\begin{theorem}[Blow-up \cite{K1}]\label{blowup_main}
If $E$ is a $\mathbb G$-Caccioppoli set in a step 2 Carnot group $\mathbb G$, $x_0\in \partial^{*}_{\mathbb G} E$  and
$\nu_E(x_0)\in H\mathbb G_{x_0}$ is the intrinsic normal then
\begin{equation}\label{blowup_convergence}
\lim_{r\to 0}{\bf 1}_{E^{\phantom {++}}_{r,x_0}}={\bf
1}_{S^+_{\mathbb G}(\nu_E(x_0))}\qquad \text{in\;} L^1_{\mathrm {loc}}(\mathbb G)
\end{equation}
and for all $R>0$
\begin{equation}\label{blowup_convergence1}
\lim_{r\to 0}\per{E_{r,x_0}}{U_c(0,R)}=
\per{S^+_{\mathbb G}(\nu_E(x_0))}{U_c(0,R)}
\end{equation}
and
$$ \per{S^+_{\mathbb G}(\nu_E(x_0))}{U_c(0,R)}= {\mathcal
H}^{n-1}(T^g_\mathbb G(\nu_E(0))\cap
U_c(0,R)).
$$
\end{theorem}

\begin{theorem}[Rectifiability Theorem \cite{K1}]\label{principale_vec} If $E\subseteq \mathbb G$ is a $\mathbb G$-Caccioppoli set
 in a step 2 Carnot group $\mathbb G$, then
\begin{gather}
\partial^{*}_{\mathbb G} E \;\text{is} \; (Q-1)\text{-dimensional}\;
\mathbb G\text{-rectifiable,} \tag{$i$}
\end{gather}
that is $\partial^{*}_{\mathbb G}{E} = N\cup\bigcup_{h=1}^{\infty}K_h$, where
${\mathcal H}^{Q-1}_{c}({N})=0$ and $K_h$ is a compact subset of a
$\mathbb G$-regular
hypersurface $S_h$;
\begin{gather}
\text{$\nu_E(p)$ is the $\mathbb G$-normal to $S_h$
at $p$, for all}\; p\in K_h;\tag{$ii$}
\end{gather}
\begin{gather}
\Per{E} = \theta_c\mathcal{S}_{c}^{Q-1}\res \partial^{*}_{\mathbb G}{E},\tag{$iii$}
\end{gather}
where $$\theta_c(x)= \frac 1 {\omega_{Q-1}}{\mathcal
H}^{n-1}\left(\partial S^+_{\mathbb G}(\nu_E(x))\cap U_c(0,1)\right).$$
Here $\omega_{k}$ is the $k$-dimensional measure of the
$k$-dimensional ball in $\mathbb R^k$. If we replace
the
$\mathcal S_c$-measure by the
$\mathcal S_\infty$-measure, the corresponding density $\theta_\infty$ is a constant. More precisely
\begin{gather} \Per{E} = \theta_\infty
\;\mathcal{S}_{\infty}^{Q-1}\res \partial^{*}_{\mathbb G}{E},\tag{$iv$}
\end{gather}
where $$\theta_\infty=\frac{\omega_{m_1-1}\omega_{m_2}
\varepsilon_2^{m_2}} {\omega_{Q-1}}=\frac 1
{\omega_{Q-1}}{\mathcal H}^{n-1}\left(\partial S^+_{\mathbb G}(\nu_E(0))
\cap U_\infty(0,1)\right).$$ Here $\varepsilon_2$ is a constant that appears in (\ref{distanzainfinito}).
\end{theorem}

A consequence of Theorem \ref{principale_vec} is the following divergence theorem.

\begin{theorem}[Divergence Theorem \cite{K1}]\label{francesco3} Let $E$ be
a $\mathbb G$-Caccioppoli set  in a step 2 Carnot group $\mathbb G$, then
\begin{equation}\Per{E} =
\theta_\infty\;\mathcal{S}_{\infty}^{Q-1}\res
\partial^{\ast}_{\mathbb G}E,\tag{$i$}
\end{equation}
and the following version of the divergence theorem holds
\begin{equation} -\int_E\mbox{div}\,_{\mathbb G}\phi\; d{\mathcal L}^n
=\theta_\infty \int_{\partial^{\ast}_{\mathbb G}E}\Scal{\nu_{E}}{\phi}
\;d{\mathcal S}^{Q-1}_{\infty},\qquad \forall \phi\in {\mathbf
C}^1_{0}(\mathbb G,H\mathbb G).\tag{$ii$}\end{equation}
\end{theorem}

\bigskip

In case the boundary of $E$ is of class $\mathbf C^1$, a sharper result holds.

\begin{corollary}[\cite{K1}]\label{Balog} If $\mathbb G$ is a Carnot group
of step 2 and a measurable set $E\subset\mathbb G$
has boundary of class ${\mathbf C}^1$ (and hence $E$ is a
$\mathbb G$-Caccioppoli set), then
\begin{equation}
\Per{E} = \theta_\infty\;\mathcal{S}_{\infty}^{Q-1}\res
\partial E=\bigg(\sum_{j=1}^{m_1}\langle
X_j,n_{E}\rangle_{\mathbb R^n}^2\bigg)^{1/2}\;{\mathcal H}^{n-1}\res \partial
E\tag{$i$}
\end{equation}
where $n_E$ denotes the Euclidean outward normal to $\partial E$.
Again a version of the
divergence theorem holds
\begin{equation} -\int_E\mbox{div}\,_{\mathbb G}\phi\; dx =\theta_\infty
\int_{\partial E}\Scal{\nu_{E}}{\phi}
\;d{\mathcal S}^{Q-1}_{\infty}\qquad \forall \phi\in {\mathbf
C}^1_{0}(\mathbb G,H\mathbb G).\tag{$ii$}\end{equation}
\end{corollary}
In order to prove this corollary, in \cite{K1} it is shown that (in a step 2 Carnot group) if $S$ is a $\mathbf C^1$-hypersurface, then $\mathcal H^{Q-1}_c(\mathcal C(S))=0$ (see \cite{K1}, Theorem 4.8). Anyway, Magnani extended this result to any Carnot group in \cite{magnani3}.
Since non-characteristic points of a boundary $\partial E$ of class $\mathbf C^1$ belongs to the reduced boundary,   $\mathcal S^{Q-1}_\infty(\partial E \setminus \partial^\ast_\mathbb G E)=0$.

\bigskip

Now, as a corollary of Theorem \ref{constant_halfspace} we can extend Theorem \ref{blowup_main}, Theorem \ref{principale_vec}, Theorem \ref{francesco3} and Corollary \ref{Balog} to our setting of Carnot groups of type $\star$.

\begin{theorem}[Blow-up]\label{blowup_new}
If $E$ is a $\mathbb G$-Caccioppoli set in a Carnot group $\mathbb G$ of type $\star$, $x_0\in \partial^{*}_{\mathbb G} E$  and
$\nu_E(x_0)\in H\mathbb G_{x_0}$ is the intrinsic normal then
\begin{equation}\label{blowup_convergence_new}
\lim_{r\to 0}{\bf 1}_{E^{\phantom {++}}_{r,x_0}}={\bf
1}_{S^+_{\mathbb G}(\nu_E(x_0))}\qquad \text{in\;} L^1_{\mathrm {loc}}(\mathbb G)
\end{equation}
and for all $R>0$
\begin{equation}\label{blowup_convergence1_new}
\lim_{r\to 0}\per{E_{r,x_0}}{U_c(0,R)}=
\per{S^+_{\mathbb G}(\nu_E(x_0))}{U_c(0,R)}
\end{equation}
and
$$ \per{S^+_{\mathbb G}(\nu_E(x_0))}{U_c(0,R)}= {\mathcal
H}^{n-1}(T^g_\mathbb G(\nu_E(0))\cap
U_c(0,R)).
$$
\end{theorem}

\begin{theorem}[Rectifiability Theorem] \label{principale} 
Let $ \mathbb G$ be a Carnot group of type $\star$.
If $E\subset \mathbb G$ is a $\mathbb G$-Caccioppoli set, then
\begin{gather}
\partial^{*}_{\mathbb G} E \;\text{is} \; (Q-1)\text{-dimensional}\;
\mathbb G\text{-rectifiable,} \tag{$i$}
\end{gather}
that is $\partial^{*}_{\mathbb G}{E} = N\cup\bigcup_{h=1}^{\infty}K_h$, where
${\mathcal H}^{Q-1}_{c}({N})=0$ and $K_h$ is a compact subset of a $\mathbb G$-regular hypersurface
 $S_h$;
\begin{gather}
\text{$\nu_E(p)$ is the $\mathbb G$-normal to $S_h$
in $p$, for every}\; p\in K_h;\tag{$ii$}
\end{gather}

\begin{gather}
\Per{E} = \theta_c\mathcal{S}_{c}^{Q-1}\res \partial^{*}_{\mathbb G}{E},\tag{$iii$}
\end{gather}
where $$\theta_c(x)= \frac 1 {\omega_{Q-1}}{\mathcal
H}^{n-1}\left(\partial S^+_{\mathbb G}(\nu_E(x))\cap U_c(0,1)\right).$$
 $\omega_{k}$ is the $k$-dimensional measure of the
$k$-dimensional ball in $\mathbb R^k$. If we replace the 
$\mathcal S_c$-measure with the
$\mathcal S_\infty$-measure, the corresponding density $\theta_\infty$ is a constant. Precisely
\begin{gather} \Per{E} = \theta_\infty
\;\mathcal{S}_{\infty}^{Q-1}\res \partial^{*}_{\mathbb G}{E},\tag{$iv$}
\end{gather}
where $$\theta_\infty=\frac{\omega_{m_1-1}\omega_{m_2}
\varepsilon_2^{m_2} \ldots  \omega_{m_\kappa} \varepsilon_\kappa^{m_\kappa}} {\omega_{Q-1}}=\frac 1
{\omega_{Q-1}}{\mathcal H}^{n-1}\left(\partial S^+_{\mathbb G}(\nu_E(0))
\cap U_\infty(0,1)\right).$$ 
Here $\varepsilon_i$ are constants that appears in (\ref{distanzainfinito}) and $\kappa$ is the step of $\mathbb G$.
\end{theorem}
The proof is the same as in \cite{K1}, but a more general value of $\theta_\infty$ is provided.

\bigskip

The following  propositions can be proved following the same arguments used in \cite{K1}.

\begin{theorem}[Divergence Theorem]\label{francesco3nuovo}
Let $ \mathbb G$ be a Carnot group of type $\star$.
If $E\subset \mathbb G$ is
a $\mathbb G$-Caccioppoli set, then
\begin{equation}\Per{E} =
\theta_\infty\;\mathcal{S}_{\infty}^{Q-1}\res
\partial^\ast_{\mathbb G}E,\tag{$i$}
\end{equation}
and the following version of the divergence theorem holds
\begin{equation} -\int_E\mbox{div}\,_{\mathbb G}\phi\; d{\mathcal L}^n
=\theta_\infty \int_{\partial^{\ast}_{\mathbb G}E}\Scal{\nu_{E}}{\phi}
\;d{\mathcal S}^{Q-1}_{\infty},\qquad \forall \phi\in {\mathbf
C}^1_{0}(\mathbb G,H\mathbb G).\tag{$ii$}\end{equation}
\end{theorem}

\begin{corollary}\label{Balognuovo}
If $\mathbb G$ is a Carnot group
of type $\star$ and a measurable set $E\subset\mathbb G$
has boundary of class ${\mathbf C}^1$ (and hence $E$ is a
$\mathbb G$-Caccioppoli set), then
\begin{equation}
\Per{E} = \theta_\infty\;\mathcal{S}_{\infty}^{Q-1}\res
\partial E=\bigg(\sum_{j=1}^{m_1}\langle
X_j,n_{E}\rangle_{\mathbb R^n}^2\bigg)^{1/2}\;{\mathcal H}^{n-1}\res \partial
E\tag{$i$}
\end{equation}
where $n_E$ denotes the Euclidean outward normal to $\partial E$.
Again a version of the
divergence theorem holds
\begin{equation} -\int_E\mbox{div}\,_{\mathbb G}\phi\; dx =\theta_\infty
\int_{\partial E}\Scal{\nu_{E}}{\phi}
\;d{\mathcal S}^{Q-1}_{\infty}\qquad \forall \phi\in {\mathbf
C}^1_{0}(\mathbb G,H\mathbb G).\tag{$ii$}\end{equation}
\end{corollary}

\section{Necessary conditions for Carnot groups of type $\star$} \label{propriet}

We do not know whether in the literature there exists an alternative characterization of stratified Lie algebra of type $\star$.
However, the following proposition gives a sufficient condition that, if satisfied, yields that a Carnot
algebra $\mathfrak g$ \textit{is not} of type $\star$, proving better insights into this condition. In particular,
it follows from Proposition \ref{quot} that
free Lie algebras of step greater than 2 and filiform stratified Lie algebras of step greater than 2
are not of type $\star$.

Unfortunately, the condition is only sufficient: see Example \ref{contrex}.

\begin{proposition}\label{quot}
Let $\mathfrak g=V_1 \oplus \ldots \oplus V_\kappa$ be a stratified Lie algebra.
We set $W_3:=\{[X_k,[X_j,X_i]] \; | \; i,j,k=1,\ldots,m  \}$, so that $V_3=\mathrm{span} (W_3)$.
The following three properties are equivalent: 
\begin{enumerate}
\item there exists a basis $(X_1,\ldots,X_m)$ for $V_1$ such that $$\mathrm{span}(W_3 \setminus \{[X_1,[X_1,X_2]],[X_1,[X_2,X_1]] \})\subsetneq V_3 \, ,$$ i.e. $[X_1,[X_1,X_2]]=-[X_1,[X_2,X_1]]$ is independent of the other commutators;

\item one of the stratified quotient algebras of $\mathfrak g $ is the Engel algebra;
 
\item  one of the stratified quotient algebras of $\mathfrak g $ is filiform of step greater than 2. 
\end{enumerate}
Moreover, if $\mathfrak g $ satisfies one of these properties, then it is not of type $\star$; in fact \textnormal{(i)} is incompatible with Lemma \ref{lemma_norip}.  
\end{proposition}
\begin{proof}
Let us show that (i)$\implies$(ii).
We take the smallest ideal $I$ that contains $V_i$ for $i=4,\ldots,r$, $X_k$ for $k=3,\ldots,m$ and $[X_2,[X_2,X_1]]$. Obviously the step of $\mathfrak g/I$ is at most 3 and the dimension of its first layer is 2. We note that $[X_1,X_2],[X_1,[X_1,X_2]] \notin I$, since  $[X_1,[X_1,X_2]]=-[X_1,[X_2,X_1]]$ is independent of the other commutators by hypothesis.    If we denote with $\pi$ the canonical projection, we have that $\pi([X_2,[X_2,X_1]])=0 $ and $\pi([X_1,[X_1,X_2]]) \neq 0$, therefore $\mathfrak g/I$ is the Engel algebra.

It's trivial to see that (ii)$\implies$(iii).

Let us see that (iii)$\implies$(i).
We denote with $ \mathfrak f$ a stratified quotient algebra of $ \mathfrak g$ that is filiform of step greater than 2 and we take a basis $(Y_1,Y_2) $ of its first layer $\tilde{V}_1$.
Since it is filiform, the dimension of its third layer is 1. Hence there exist $(a,b)\neq (0,0)$ in $\mathbb R^2$ such that
\begin{equation}\label{dip} a[Y_1,[Y_1,Y_2]]+b[Y_2,[Y_2,Y_1]]=0 \end{equation}
If $a=0$ it becomes $[Y_2,[Y_2,Y_1]]=0 $, whereas if $b=0$ it can be reduced to the same case by exchanging the role of $Y_1$ and $Y_2$.
If $a\neq 0$ and $b \neq 0$, then we perform the change of basis
\begin{equation}\label{replacing}
\begin{cases} Y_1=b \tilde Y_1 \\  Y_2=a \tilde Y_1+\tilde Y_2
\end{cases}
\end{equation}
By replacing (\ref{replacing}) in (\ref{dip}), we obtain $[\tilde Y_2,[\tilde Y_2,\tilde Y_1]]=0 $.
Therefore, up to a change of basis, $[Y_2,[Y_2,Y_1]]=0$.

Now we consider the canonical projection $\mathfrak g \overset{\pi}{\longrightarrow} \mathfrak f $, which is a surjective homogeneous homomorphism.
Hence, there exist two independent vectors $X_1, X_2$ in $V_1$ such that $\pi(X_1)=Y_1$ and $\pi(X_2)=Y_2$. We make a basis for $V_1$ that contains $X_1$ and $X_2$ and such that its other elements $X_i$ with $i=3,\ldots,m$ are taken from a basis of $\mathrm{Ker}\,\pi_{|V_{1}}$. This can be done because of linear algebra arguments applied to the surjective linear function $\pi_{|V_{1}}:V_1\rightarrow \tilde{V}_1$.
 
We have that $$\mathrm{span}(W_3 \setminus \{[X_1,[X_1,X_2]],[X_1,[X_2,X_1]] \})\subsetneq V_3,$$ otherwise $[Y_1,[Y_1,Y_2]]$ would be zero and $\mathfrak f$ would be of step 2, which is a contradiction. 
\end{proof}

%
%

\begin{example}\label{contrex}
Denote by $\mathfrak{f}_{m,\kappa}$ the free Lie algebra of step $\kappa$ with $m$ generators, and  take a free Lie algebra $\mathfrak{f}_{3,\kappa}$ with $ \kappa\geq 3 $ and three generators $X_1,X_2,X_3$. We quotient it by the homogeneous ideal generated by \begin{align} \label{relations} \nonumber \big\{&[X_1,[X_1,X_2]]+[X_1,[X_1,X_3]],\,[X_1,[X_1,X_3]]+[X_2,[X_2,X_1]], \\ 
&[X_2,[X_2,X_1]]+[X_2,[X_2,X_3]],\,
[X_2,[X_2,X_3]]+[X_3,[X_3,X_1]], \\ \nonumber
&[X_3,[X_3,X_1]]+[X_3,[X_3,X_2]]\big\}, \end{align}
and we denote the obtained stratified quotient algebra by $\mathfrak g$.
We stress that vector fields of (\ref{relations}) are linearly independent, since Jacobi identity is trivial in those cases. 
By recalling Remark \ref{dimensionreason}, we can say that the dimension of the third layer of $\mathfrak{f}_{3,\kappa}$ is $8$, whereas the dimension of the third layer of $ \mathfrak g$ is $3$ because of (\ref{relations}).
In a stratified Lie algebra of type $\star$ with $3$ generators,  the dimension of the third layer is at most $2$, hence the stratified quotient algebra is not of type $\star$.
Now we show by contradiction that (i) of Proposition \ref{quot} does not hold. 

We assume there exists a basis of the first layer of $\mathfrak g$, denoted by $(Y_1,Y_2,Y_3)$, such that $$[Y_1,[Y_1,Y_2]]=-[Y_1,[Y_2,Y_1]]$$ is independent of the other commutators. If we rewrite the vectors that generate the ideal
as linear combinations of commutators of $Y_1,Y_2,Y_3$, we obtain
\begin{align}\nonumber \Big\{&\alpha_{1}[Y_1,[Y_1,Y_2]]+\sum_i \beta_{1i}Z_i,\,\alpha_{2}[Y_1,[Y_1,Y_2]]+\sum_i \beta_{2i}Z_i, \\ \nonumber
&\alpha_{3}[Y_1,[Y_1,Y_2]]+\sum_i \beta_{3i}Z_i,\,
\alpha_{4}[Y_1,[Y_1,Y_2]]+\sum_i \beta_{4i}Z_i, \\
&\alpha_{5}[Y_1,[Y_1,Y_2]]+\sum_i \beta_{5i}Z_i \Big\}, \label{relations2}\end{align}
where $Z_i$ are the remaining commutators of length 3, excluding $[Y_1,[Y_2,Y_1]]$.
We remark that, since our assumptions, \begin{equation}\label{sonozero} \alpha_i=0 \; \textrm{ for } \; i=1,\ldots,5 . \end{equation}
If
\begin{displaymath}
\left(\!
\begin{array}{c}
X_1 \\
X_2 \\
X_3 \\
\end{array}\!\right)=\left(\!
\begin{array}{ccc}
a_{11} & a_{12} & a_{13} \\
a_{21} & a_{22} & a_{23} \\
a_{31} & a_{32} & a_{33} \\
\end{array}\!\right)\left(\!
\begin{array}{c}
Y_1 \\
Y_2 \\
Y_3 \\
\end{array}\!\right)
\end{displaymath}
is the relation between the two basis, we can replace coefficients $\alpha_i$ in (\ref{sonozero}) with their expression and we obtain \begin{displaymath}
\begin{cases}\;\;\:a_{11}(-a_{12} a_{21} + a_{11} a_{22}) + a_{11} (-a_{12} a_{31} + a_{11} a_{32})=0 \\
 -a_{21}(-a_{12} a_{21} + a_{11} a_{22}) + a_{11}(-a_{12} a_{31} + a_{11} a_{32})=0 \\
  -a_{21}(-a_{12} a_{21} + a_{11} a_{22}) + a_{21}(-a_{22} a_{31} + a_{21} a_{32})=0 \\
 -a_{31} (-a_{12} a_{31} + a_{11} a_{32}) + a_{21}(-a_{22} a_{31} + a_{21} a_{32})=0 \\
 -a_{31} (-a_{12} a_{31} + a_{11} a_{32}) - a_{31}(-a_{22} a_{31} + a_{21} a_{32})=0 \; .
\end{cases}
\end{displaymath}
This system of equations leads to the following solutions:
\begin{gather*}
a_{11} = 0 \textrm{ and } a_{21} = 0 \textrm{ and } a_{31} = 0 \\  \textrm{ or } \\
      a_{12} = 0 \textrm{ and } a_{22} = 0 \textrm{ and } a_{32} = 0 \\ \textrm{ or } \\
      a_{11} = 0 \textrm{ and } a_{12} = 0 \textrm{ and } a_{21} = 0  \textrm{ and } a_{22} = 0 \\ \textrm{ or } \\ 
      a_{11} = 0 \textrm{ and } a_{12} = 0 \textrm{ and } a_{31} = 0 \textrm{ and } a_{32} = 0 \\ \textrm{ or } \\
      a_{21} = 0 \textrm{ and } a_{22} = 0 \textrm{ and } a_{31} = 0 \textrm{ and } a_{32} = 0 \\ \textrm{ or } \\
      a_{21} = 0 \textrm{ and } a_{22} = 0 \textrm{ and } a_{11} = \frac{a_{12} a_{31}}{a_{32}} \textrm{ and } 
        a_{32} \neq 0 \\ \textrm{ or }\\
      a_{11} = \frac{a_{12} a_{21}}{a_{22}} \textrm{ and } a_{22} \neq 0 \textrm{ and } a_{31} = 0 \textrm{ and } 
        a_{32} = 0 \\ \textrm{ or } \\
      a_{11} = \frac{a_{12} a_{21}}{a_{22}} \textrm{ and } a_{22} \neq 0 \textrm{ and } 
        a_{21} = \frac{a_{22} a_{31}}{a_{32}} \textrm{ and } a_{32} \neq 0 \\ \textrm{ or } \\
      a_{11} = 0 \textrm{ and } a_{12} =0 \textrm{ and } a_{21} \neq 0 \textrm{ and } a_{21} = \frac{a_{22} a_{31}}{a_{32}} \textrm{ and } 
        a_{32} \neq 0 \;  .     
\end{gather*}

In any of these cases, the change of basis matrix is singular, which is a contradiction.  
\end{example}

\bigskip

We conclude this section with some examples of groups in which the blowup of a set (at a point of its reduced boundary) is not always a halfspace. 
\begin{example}\label{nofree}
Let $\mathbb{G}$ be a free Carnot group of step $\kappa >2$ with $m$ generators $(m \geq 2)$. Then Blow-up Theorem does not hold. 

Theorem 14.1.10 of \cite{lanco_bonfi_ugu} gives a model for its Lie algebra $\mathfrak{g}$ in terms of $m$ generating vector fields with polynomial coefficients on $\mathbb R^n$, where $n$ is the dimension of $\mathfrak{g}$. By Remark 14.1.11 of \cite{lanco_bonfi_ugu}, these $m$ vector fields naturally define a free Carnot group of step $\kappa$ and $m$ generators: more precisely, they are left invariant vector fields of a Carnot group $(\mathbb R^n,\circ,\delta_\lambda)$ that is isomorphic to $\mathbb G$. In general, the coordinate system given by this isomorphism is not the exponential one defined in (\ref{esp}). We denote these generating vector fields with $X_1,X_2,\ldots,X_m$. 

By Theorem 14.1.10 of \cite{lanco_bonfi_ugu}, we have that 
$$ X_1=\frac{\partial}{\partial x_1},\quad X_2=\frac{\partial}{\partial x_2}+\ldots+\frac{x_1^{2}}{2} \frac{\partial}{\partial x_j}+\ldots+a_{2,n}(x_1,x_2,\ldots,x_n) \frac{\partial}{\partial x_n}, $$
where $j$ represent the position of $[[X_2,X_1],X_1]$ in the Hall basis for $\mathfrak{f}_{m,\kappa}$.
Moreover, in $X_3,\!\ldots\!,\!X_m$, the partial derivative $\frac{\partial}{\partial x_j}$ does not appear.

\medskip

Let $E=\{x\in\mathbb G\, :\, f(x) \ge 0\}$, where
$$
f(x_1,\ldots,x_n)=\frac{x_2^3}{3}+2x_j.
$$
We note that $\partial E=\{x\in\mathbb R^4\, :\, f(x) = 0\}$ is a smooth Euclidean
manifold, hence $E$ is a $\mathbb G$-Caccioppoli set (see \cite{K1}, Proposition 2.22).
We stress that $\partial E$ is not a vertical hyperplane of $\mathbb G$.

The horizontal gradient of $f$ is
$
\nabla_{\mathbb G}f(x) = \left(0,x_1^2+x_2^2\right)$ and the intrinsic normal is
$$
\nu_E(x)=-\frac{\nabla_{\mathbb G}f(x)}{|\nabla_{\mathbb G}f(x)|}=(0,-1)$$
for every $x\in\partial E\setminus N$, where $N=\{x\in \mathbb E\,:\,
x_1=x_2=0\}$.
Here we used point $(iii)$ of  Theorem 2.1 (Implicit Function Theorem) in \cite{FSSC dini}.
 Since $|\partial E|_{\mathbb G}(N)=0$, the origin
belongs to $\partial^{*}_{\mathbb G} E$. We note that
$f(\delta_{\lambda}x)=\lambda^3 f(x)$ for $\lambda>0$, hence
$ E_{\lambda,0}=\delta_{\lambda}E=E$. Finally we can conclude that (\ref{blowup_convergence})
is false since $E$ is not a vertical halfspace.
\end{example}

\begin{example}
Let $\mathbb{G}$ be a filiform Carnot group whose Lie algebra is of type
\begin{align*}
&\textrm{span}\{ X_1,X_2 \} \oplus \textrm{span}\{ [X_2,X_1]\} \oplus \textrm{span}\{ [[X_2,X_1],X_1] \} \oplus \cdots \\ &\cdots \oplus \textrm{span} \{[[\cdots[[X_2,\underbrace{X_1],X_1],\cdots],X_1}_{(\kappa-1) \textrm{ times}}] \}
\end{align*}
where $\kappa>2$ and all other independent commutators are identically zero.
Then Blow-up Theorem does not hold. 
A coordinate expression of vector fields $X_1$ and $X_2$ is
$$ X_1=\frac{\partial}{\partial x_1}, \quad X_2=\frac{\partial}{\partial x_2}-x_1\frac{\partial}{\partial x_3}+\frac{x_1^2}{2}\frac{\partial}{\partial x_4}+\ldots+(-1)^{(\kappa-1)} \frac{x_1^{\kappa-1}}{(\kappa-1)!}\frac{\partial}{\partial x_{\kappa+1}}. $$

 Let $E=\{x\in\mathbb G\, :\, f(x) \ge 0\}$, where
$$
f(x_1,\ldots,x_{r+1})=\frac{x_2^3}{3}+2x_4.
$$
Then the proof follows the same argument used in the previous example. 
\end{example}



\begin{thebibliography}{99}






\bibitem{newambro} Ambrosio, L., Kleiner, B., Le Donne, E.: Rectifiability of Sets of Finite Perimeter in Carnot Groups: Existence of a Tangent Hyperplane. J. Geom. Anal. {\bf 19}, 509--540 (2009)

\bibitem{lanco_bonfi_ugu}  Bonfiglioli, A., Lanconelli, E., Uguzzoni, F.:  Stratified Lie Groups and Potential Theory for their Sub-Laplacians. Springer, New York (2007)


\bibitem{degiorgi} De Giorgi, E.: Nuovi teoremi relativi alle misure $(r-1)$-dimensionali in uno spazio ad $r$ dimensioni. Ricerche Mat. \textbf{4}, 95--113 (1955)



\bibitem{K2}  Eves, H.W.:  Elementary matrix theory. Courier Dover Publications, New York  (1980)


\bibitem{folland} Folland, G.B.:  Subelliptic estimates and
function spaces on nilpotent Lie groups. Ark. Mat. \textbf{13}, 161--207 (1975)

\bibitem{FS} Folland, G.B., Stein, E.M.:  Hardy spaces on
homogeneous groups. Princeton University Press, Princeton (1982)


\bibitem{FSSC1} Franchi, B., Serapioni, R., Serra Cassano, F.: 
Meyers-Serrin type theorems and relaxation of variational integrals
depending on vector fields. Houston J. Math. {\bf 22}(4),
859--889 (1996)



\bibitem{FSSC dini} Franchi, B., Serapioni, R., Serra Cassano, F.:
 Regular hypersurfaces, intrinsic perimeter and implicit
function theorem in Carnot groups. Comm. Anal. Geom. \textbf{11}(5), 909--944 (2003) 


\bibitem{K1}  Franchi, B., Serapioni, R., Serra Cassano, F.:  On the Structure of Finite Perimeter Sets in Step 2 Carnot Groups. J.  Geom. Anal. \textbf{13},  421--466 (2003)


\bibitem{garnie}  Garofalo, N., Nhieu, D.M.: Isoperimetric and
Sobolev inequalities for Carnot-Carath\'eodory spaces and the existence of
minimal surfaces. Comm. Pure Appl. Math. \textbf{49}, 1081--1144 (1996) 


\bibitem{unitupperbook} Gorbatsevich, V.V., Onishchik, A.L., Vinberg, E.B.: Foundations of Lie Theory and Lie Transformation Groups. Springer, Berlin (1997) 




\bibitem{koranireim} Kor\'anyi, A., Reimann, H.M.:  Foundation for
the theory of quasiconformal mappings on the Heisenberg group.
Advances in Mathematics \textbf{111}, 1--87 (1995) 



\bibitem{magnani3}  Magnani, V.:  Characteristic points, rectifiability and perimeter measure on stratified groups. J. Eur. Math. Soc. \textbf{8}(4), 585--609 (2006)


\bibitem{Mi} Mitchell, J.:  On Carnot-Carath\'eodory metrics.
J. Differ. Geom. {\bf 21},  35--45 (1985)

\bibitem{subriemann}  Montgomery, R.:  A Tour of
Subriemannian Geometries, Their Geodesics and Applications.
Mathematical Surveys and Monographs {\bf 91},
AMS, Providence RI (2002)

\bibitem{ballmetrics} Nagel, A., Stein,  E.M., Wainger, S.: 
Balls and metrics defined by vector fields I: Basic properties.
Acta Mathematica \textbf{155}(1), 103--147 (1985)

\bibitem{pansu}  Pansu, P.:  M\'etriques de Carnot-Carath\'eodory et
quasiisom\'etries des espaces sym\'etriques de rang un. Ann. of
Math. \textbf{129},  1--60 (1989)


\bibitem{varsalcoul}  Varopoulos, N.Th., Saloff-Coste, L., Coulhon, T.:
Analysis and Geometry on Groups. Cambridge University Press,
Cambridge (1992)




\end{thebibliography}
\end{document}